\documentclass[11pt]{article}

\usepackage{latexsym,mathrsfs}
\usepackage{amsmath,amssymb}
\usepackage{amsthm,enumerate,verbatim}
\usepackage{amsfonts}
\usepackage{graphicx}
\usepackage{algorithm}
\usepackage{algorithmic}
\usepackage{url}

\renewcommand{\algorithmicrequire}{\textbf{Input:}}
\renewcommand{\algorithmicensure}{\textbf{Output:}}

\newtheorem{lemma}{Lemma}

\newtheorem{theorem}{Theorem}

\newtheorem*{definition*}{Definition}

\newtheorem{assumption}{Assumption}
\newtheorem{remark}{Remark}

\title{Multiplicative Updates for Polynomial Root Finding}
\date{}

\author{Nicolas Gillis  \\ 
Department of Mathematics and Operational Research \\ 
Facult\'e Polytechnique, Universit\'e de Mons \\ 
Rue de Houdain 9, 7000 Mons, Belgium\\
 nicolas.gillis@umons.ac.be  
}

\begin{document}

\maketitle

\begin{abstract}
Let $f(x)=p(x)-q(x)$ be a polynomial with real coefficients whose roots have nonnegative real part, 
where $p$ and $q$ are polynomials with nonnegative coefficients. In this paper, we prove the following: Given an initial point $x_0 > 0$, the multiplicative update  $x_{t+1} = x_t \, p(x_t)/q(x_t)$ ($t=0,1,\dots$) monotonically and linearly converges to the largest  (resp.\@ smallest) real roots of $f$ smaller (resp.\@ larger) than $x_0$ if $p(x_0) < q(x_0)$ (resp.\@  $q(x_0) < p(x_0)$). The motivation to study this algorithm comes from the multiplicative updates proposed in the literature to solve optimization problems with nonnegativity constraints; in particular many variants of nonnegative matrix factorization. 
\end{abstract}


\section{Introduction}

Let $f(x) = p(x) - q(x)$ be a polynomial where $p$ and $q$ have nonnegative coefficients. 
We would like to compute a root of $f$, that is, find $x$ such that $f(x) = 0 \iff p(x) = q(x)$.  
Let $x_0 \in \mathbb{R}$ with $x_0 > 0$ (the same idea can be used if $x_0$ is negative), 
and let us denote $r_1$ (resp.\@ $r_m$) the smallest (resp.\@ largest) nonnegative real root of $f$. 
Let us also define 
\[
\underline{r} = \left\{ 
\begin{array}{cc} 
\text{the largest real root of $f$ smaller than $x_0$} & \text{if } x_0 \geq r_1,  \\ 
0 & \text{otherwise,} 
\end{array} \right. 
\] 
and 
\[
\bar{r} = \left\{ 
\begin{array}{cc} 
\text{the smallest real root of $f$ larger than $x_0$} & \text{if } x_0 \leq r_m,  \\ 
+\infty & \text{otherwise, } 
\end{array} \right.  
\]  
such that $x_0 \in [\underline{r},\bar{r}]$. Note that if $x_0$ is equal to a root of $f$, then $x_0=\underline{r}=\bar{r}$.  
The point $x_0$ is a root of $f$ if and only if $p(x_0) = q(x_0)$, otherwise one may apply the multiplicative updates 
$x_0 \frac{p(x_0)}{q(x_0)}$ and $x_0 \frac{q(x_0)}{p(x_0)}$ to generate new points that hopefully get closer to roots of $f$. The intuition is that the roots of $f$ are fixed points of these updates.   
Suppose without loss of generality that $p(x_0) > q(x_0)$. Then, we have 
\begin{displaymath}
x_{-1} = x_0 \frac{q(x_0)}{p(x_0)} \;  < \; x_0 \; < \; x_1 = x_0 \frac{p(x_0)}{q(x_0)}.
\end{displaymath}
Two points have been generated: $x_1$ greater than $x_0$ and the $x_{-1}$ smaller than $x_0$. 
In this paper, we will prove that, under some assumptions, 
$x_1$ and $x_{-1}$ belong to the same interval as $x_0$, that is, 
\[ 
\underline{r} \;\leq\; x_{-1} \;<\; x_{0} \;<\; x_{1} \;\leq\; \bar{r}, 
\] 
so that applying the above multiplicative updates iteratively generates two sequences converging monotonically to $\underline{r}$ and $\bar{r}$ (Theorems~\ref{MainTh} and~\ref{convr}). We will also prove that this algorithm has local linear convergence for simple roots (Theorem~\ref{thconv}). \\ 

The motivation to study the above updates comes from the paper~\cite{sha2007multiplicative} where such multiplicative updates are used to solve quadratic programs with nonnegativity constraints, 
and from the literature on nonnegative matrix factorization (NMF) where such updates are used extensively to find solutions of the first-order optimality conditions; 
see for example~\cite{cichocki2009nonnegative} and the references therein.  
The popularity of these multiplicative updates in the NMF literature comes from the fact that 
(1)~they were the algorithm proposed in~\cite{lee1999learning, lee2001algorithms} that launched the research on NMF, 
(2)~they are rather simple to derive and implement, and 
(3)~there is no parameter to tune.   
However, they usually converge slower than more sophisticated techniques such as coordinate descent methods; see, e.g.,~\cite{G14a}.  

The goal and main contribution of this paper is to get more insight on such multiplicative updates by proving their convergence for univariate polynomials. It is organized as follows. 
Section~\ref{not} gives the assumptions and the notation used throughout the paper, 
Section~\ref{mainres} proves the convergence of the multiplicative updates as outlined above, 
and Section~\ref{examp} provides a numerical example.

\section{Assumptions and Notation} \label{not}

Let us write the polynomial $f$ of degree $n$ as follows
\begin{equation*}
f(x) = \sum_{i=0}^n (-1)^i a_{n-i} x^{n-i}, \quad \text{ where } 
a_{n-i} \in \mathbb{R} \text{ for }  1 \leq i \leq n, 
 \end{equation*} 
and where we assume $a_n = 1$ without loss of generality. 

\begin{assumption}
\label{as1}
The real parts of the roots of $f$ are nonnegative, and $f$ has at least one root with positive real part. 
\end{assumption}
If Assumption~\ref{as1} is not satisfied, 
one can shift the polynomial, that is, $f(x) \leftarrow f(x-x_0)$ for some real $x_0$ sufficiently large.  
The polynomial $f$ can be split as the difference of two polynomials with nonnegative coefficients as follows:
\begin{equation} \label{unipoly} 
f(x) = \sum_{i=0}^n (-1)^i a_{n-i} x^{n-i} = p(x) - q(x),  
\end{equation}  
where 
\begin{equation} \label{pxnx}   
p(x)  =  \sum_{i=0}^{\left\lceil (n-1)/2 \right\rceil} a_{n-2i} \, x^{n-2i} 
\quad \text{ and } \quad 
q(x)  =  \sum_{i=0}^{\left\lfloor (n-1)/2 \right\rfloor} a_{n-(2i+1)} \, x^{n-(2i+1)} . 
\end{equation} 
Defining $a_{j} = 0$ for all $j \notin \{0,1,\dots,n\}$, we also have 
\[
q(x)  =  \sum_{i=0}^{\left\lceil (n-1)/2 \right\rceil} a_{n-(2i+1)} \, x^{n-(2i+1)}, 
\]  
so that $q$ and $p$ sum over the same indices, which will be useful later.  
Let us denote 
\begin{equation*}
r_1 \leq r_2 \leq \dots \leq r_m
\label{roots}
\end{equation*}
the real roots of $f$ in nondecreasing order. 
Let us also denote $r_0=0$, $r_{m+1} = +\infty$ and $r_i$ the complex roots of $f$ for $m+2 \leq i \leq n+1$,  
and $I = \{1,2,\dots,m,m+2,\dots,n+1\}$ the indices of the roots of $f$. 
Therefore, we have $f(x) = \prod_{i \in I} (x-r_i)$,  $a_n = 1$, and 
\begin{equation}  \label{anj}
a_{n-j} = 
\sum_{J \subset I, |J|=j} \; \prod_{i \in J} r_i 
\quad \text{ for } 1 \leq j \leq n. 
\end{equation} 
Under Assumption~\ref{as1}, the coefficients of a polynomial $f$ are alternating, 
that is, 
$a_i \geq 0$ for $0 \leq i \leq n$, since Re$(r_i) \geq 0$ for all $i \in I$. 
In fact, all real roots are nonnegative while, for the complex roots, 
we have the following result. 
\begin{lemma}
Let $\mathcal{Z} = \cup_{i=1}^k \{ z_i, \bar{z}_i \}$ be a set of $k$ complex numbers and their conjugates with nonnegative real parts. Then, for any $1 \leq j \leq 2k$, 
\[
f(\mathcal{Z},j) = \sum_{J \subset \mathcal{Z}, |J|=j} \; \prod_{z_i \in J} z_i 
\quad 
\text { is a nonnegative real number.}
\]  
\end{lemma} 
\begin{proof}
We prove the result by induction on $k$ ($|\mathcal{Z}|$ contains $2k$ elements).  

\noindent \textit{Case $k=1$.} For $\mathcal{Z} = \{ z, \bar{z} \}$, we have 
$f(\mathcal{Z},1) = z+\bar{z} = 2 \text{Re}(z)$, and 
$f(\mathcal{Z},2) = z\bar{z} = |z|^2$. 

\noindent \textit{Induction.} Let $\mathcal{Z} = \mathcal{Z}' \cup \{ z, \bar{z} \}$. 
We have 
\begin{align*}
f(\mathcal{Z},j) 
& = z f(\mathcal{Z}',j-1) + \bar{z} f(\mathcal{Z}',j-1) 
+ z \bar{z} f(\mathcal{Z}',j-2) + f(\mathcal{Z}',j)  \\ 
& =  2 \text{Re}(z)f(\mathcal{Z}',j-1) + |z|^2 f(\mathcal{Z}',j-2)  + f(\mathcal{Z}',j). 
\end{align*}
where $\mathcal{Z}'$ contains $2k-2$ elements. 
\end{proof} 

Moreover, $p(x) > 0$ and $q(x) >0$ for all $x > 0$ since $p$ and $q$ have at least one positive coefficient since $f$ has at least one root with positive real part.  
For $1 \leq k \leq m$, let us define $a_{(n-j,k)}$ as follows 
\begin{equation} \label{anjk} 
a_{(n-j,k)} :=
\left\{ \begin{array}{cc}
\sum_{J \subset I, |J|=j, k \notin J} \; \prod_{i \in J} r_i 
& 1 \leq j \leq n,  \\
1 & j = 0,   \\
0 & \textrm{ otherwise.}   
\end{array} \right.
\end{equation} 
We have that $a_{(n-j,k)}$ is the sum of the same terms as $a_{n-j}$ in~\eqref{anj} except the ones where the $k$th root of $f$ appears. This implies that 
\begin{equation} \label{anjk2}
a_{(n-j,k)} =  a_{n-j} - r_k a_{(n-(j-1),k)}.  
\end{equation} 
In fact, $a_{(n-(j-1),k)}$ is the sum of all the products of $j-1$ roots of $f$ except for $r_k$. 
Note that $a_{(n-j,k)} \geq 0$ for all $j,k$ for a polynomial $f$ satisfying Assumption~\ref{as1} (for the same reasons as for $f$, since we only allow $r_k$ to be a real root with $1\leq k \leq m$).   
For all $j$, $k$, let us show that 
$a_{n-(j+1)} - r_k a_{n-j} =  a_{(n-(j+1),k)} - r_k^2 a_{(n-(j-1),k)}$.  
Using~\eqref{anjk2}, we obtain 
\begin{align}
a_{n-(j+1)} - r_k a_{n-j} 
& =  \left( a_{(n-(j+1),k)} + r_k a_{(n-j,k)} \right)  - r_k \left( a_{(n-j,k)} + r_k a_{(n-(j-1),k)}  \right) \nonumber \\ 
& = a_{(n-(j+1),k)} - r_k^2 a_{(n-(j-1),k)}. \label{lem1}  
\end{align}

\section{Main result} \label{mainres}

We can now prove our main result. 
\begin{theorem} \label{MainTh}
Let $f(x)$ be a univariate polynomial of degree $n$ defined as in~\eqref{unipoly} and satisfying Assumption \ref{as1}, and let $p(x)$ and $q(x)$ be defined as in~\eqref{pxnx}. 
Let also $x_0 \in \mathbb{R}$ with $0 < x_0 \in [r_k,r_{k+1}]$ for some $k \in \{0,1,\dots,m\}$. 
Then, 
\begin{displaymath}
x_1 = x_0 \frac{p(x_0)}{q(x_0)} \in [r_k,r_{k+1}] \quad \textrm{ and } \quad x_{-1} = x_0 \frac{q(x_0)}{p(x_0)} \in [r_k,r_{k+1}].  
\end{displaymath}
\end{theorem} 
\begin{proof}
Since $p(x)$ and $q(x)$ only intersect at the roots of $f$, we have 
\begin{eqnarray}
p(x_0) \geq q(x_0) & \Rightarrow & p(x) \geq q(x) \; \text{ for } \; x \in [r_k,r_{k+1}], \label{pn1}   
\end{eqnarray}
and similarly for $p(x_0) \leq q(x_0)$. 
Let us focus on the case $p(x_0) \geq q(x_0)$. The case  $p(x_0) \leq q(x_0)$ can be treated in a similar way. 
Clearly, by \eqref{pn1}, $x_0 \frac{p(x_0)}{q(x_0)} \geq x_0 \geq  r_k$ and $x_0 \frac{q(x_0)}{p(x_0)} \leq x_0 \leq r_{k+1}$. It remains to show that 
(i)~$x_0 \frac{q(x_0)}{p(x_0)}  \geq  r_{k}$, 
and (ii)~$x_0 \frac{p(x_0)}{q(x_0)}  \leq  r_{k+1}$. 
Let us start with (i). We have 
\begin{eqnarray}
%
%
%
x_0 \frac{q(x_0)}{p(x_0)}  \geq  r_{k} & \iff & x_0 q(x_0) - r_k p(x_0) \geq 0 \nonumber \\
& \iff & \left[\sum_{i=0}^{\left\lceil (n-1)/2 \right\rceil} a_{n-(2i+1)} \, x_0^{n-2i} \right] 
-  \left[\sum_{i=0}^{\left\lceil (n-1)/2 \right\rceil} r_k a_{n-2i} \, x_0^{n-2i}\right] \geq 0 \nonumber  \\
& \iff & \alpha := 
\sum_{i=0}^{\left\lceil (n-1)/2 \right\rceil} (a_{n-(2i+1)}-r_k a_{n-2i})  \, x_0^{n-2i} \geq 0. \label{alpha1}
\end{eqnarray} 
Therefore, it remains to prove that $\alpha$ is nonnegative. 
Using~\eqref{lem1}, the fact that $x_0 \geq r_k \geq 0$ and $a_{n-j,k} \geq 0$ for all $j,k$, we obtain  
\begin{equation}
a_{n-(2i+1)}-r_k a_{n-2i}  =  a_{(n-(2i+1),k)} - r_k^2 a_{(n-(2i-1),k)} \geq  a_{(n-(2i+1),k)} - x_0^{2} a_{(n-(2i-1),k)}. \label{lowalpha}
\end{equation}
Replacing the expression in brackets in~\eqref{alpha1} by the right-hand side of~\eqref{lowalpha}, we get a lower bound for $\alpha$:
\begin{align*}
\alpha 
 & \geq 
 \sum_{i=0}^{\left\lceil (n-1)/2 \right\rceil} \Big( a_{(n-(2i+1),k)} - x_0^2 a_{(n-(2i-1),k)} \Big)  \, x_0^{n-2i} \\ 
& =  
 \sum_{i=0}^{\left\lceil (n-1)/2 \right\rceil}  a_{(n-(2i+1),k)} \, x_0^{n-2i} 
- \sum_{i=1}^{\left\lceil (n-1)/2 \right\rceil}  a_{(n-(2i-1),k)}  \, x_0^{n-2(i-1)} \\
 & =  \sum_{i=0}^{\left\lceil (n-1)/2 \right\rceil}  a_{(n-(2i+1),k)} \, x_0^{n-2i} 
- \sum_{j=0}^{\left\lceil (n-1)/2 \right\rceil - 1 }  a_{(n-(2j+1),k)}  \, x_0^{n-2j} \\
 & =  \sum_{i=0}^{\left\lceil (n-1)/2 \right\rceil - 1} \underbrace{\Big( a_{(n-(2i+1),k)} - a_{(n-(2i+1),k)} \Big)}_{=0} 
\, x_0^{n-2i} + \gamma_n x_0 \geq 0 ,  
\end{align*} 
where 
\[
\gamma_n = 
\left\{ 
\begin{array}{cc} 
0 & \text{ if $n$ is even since $a_{(n-(2i+1),k)} = a_{(-1,k)} = 0$ for $i=\lceil (n-1)/2\rceil=n/2$},  \\  
1 & \text{ if $n$ is odd since $a_{(n-(2i+1),k)} = a_{(0,k)} = 1$ for $i=\lceil (n-1)/2\rceil=(n-1)/2$}. 
\end{array}  
\right.  
\]
The first equality follows from the fact that $a_{(n+1,k)}=0$ by definition~\eqref{anjk}, 
the second simply by setting $j=i-1$, 
and the third by putting back the terms together.  

Let us now focus on (ii). The proof is rather similar to (i) 
but we provide it here for completeness. We have 
$x_0 \frac{p(x_0)}{q(x_0)}  \leq  r_{k+1}  \iff  r_{k+1} q(x_0) - x_0 p(x_0) \geq 0$ 
\begin{eqnarray}
& \iff & \left[\sum_{i=0}^{\left\lceil (n-1)/2 \right\rceil} r_{k+1} a_{n-(2i+1)} \, x_0^{n-(2i+1)} \right] 
-  \left[\sum_{i=0}^{\left\lceil (n-1)/2 \right\rceil} a_{n-2i} \, x_0^{n-(2i-1)}\right] \geq 0 \nonumber \\
& \iff & \left[\sum_{j=i+1=1}^{\left\lceil (n-1)/2 \right\rceil + 1} r_{k+1} a_{n-(2j-1)} \, x_0^{n-(2j-1)} \right] 
-  \left[\sum_{i=0}^{\left\lceil (n-1)/2 \right\rceil} a_{n-2i} \, x_0^{n-(2i-1)}\right] \geq 0 \nonumber  \\
& \iff & \beta := \sum_{i=0}^{\left\lceil (n-1)/2 \right\rceil+1} (r_{k+1} a_{n-(2i-1)}-a_{n-2i})  \, x_0^{n-(2i-1)} \geq 0, \label{beta}
\end{eqnarray} 
where $a_{n-(2i-1)} = a_{n+1} = 0$ for $i = 0$, and $a_{n-2i} = 0$ for $i = \left\lceil (n-1)/2 \right\rceil+1$.   
Using~\eqref{lem1} multiplied by $-1$, and since $0 \leq x_0 \leq r_{k+1}$ and $a_{n-j}^k \geq 0$ for all $j$, we obtain 
\begin{equation}
r_{k+1} a_{n-(2i-1)} - a_{n-2i} = r_{k+1}^2 a_{(n-(2i-2),k+1)} - a_{(n-2i,k+1)}  \geq  x_0^2 a_{(n-(2i-2),k+1)} - a_{(n-2i,k+1)}. \label{betalow}
\end{equation}
Similarly as for (i), injecting~\eqref{beta} in~\eqref{betalow}, we can lower bound $\beta$ as follows  
\begin{eqnarray*}
\beta 
& \geq & 
\sum_{i=0}^{\left\lceil (n-1)/2  \right\rceil + 1} 
\Big(x_0^2 a_{(n-(2i-2),k+1)} - a_{(n-2i,k+1)}\Big)  \, x_0^{n-(2i-1)} \\
& = & 
\sum_{i=1}^{\left\lceil (n-1)/2  \right\rceil + 1}  a_{(n-(2i-2),k+1)} x_0^{n-(2i-3)}
- \sum_{i=0}^{\left\lceil (n-1)/2  \right\rceil} a_{(n-2i,k+1)} x_0^{n-(2i-1)}   \\
& = & \sum_{j=i-1=0}^{\left\lceil (n-1)/2  \right\rceil}  a_{(n-2j,k+1)} x_0^{n-(2j-1)}
- \sum_{i=0}^{\left\lceil (n-1)/2  \right\rceil} a_{(n-2i,k+1)} x_0^{n-(2i-1)}   \\ 
& = &  \sum_{i=0}^{\left\lceil (n-1)/2  \right\rceil} 
\Big( a_{(n-2i,k+1)} - a_{(n-2i,k+1)} \Big) x_0^{n-(2i-1)}  = 0.    
\end{eqnarray*} 
\end{proof} 

Theorem~\ref{MainTh} allows us to construct two sequences converging monotically to $r_k$ and $r_{k+1}$, given an initial point 
$r_k < x_0 < r_{k+1}$; see Algorithm~\ref{muroot}.

 \algsetup{indent=2em}
\begin{algorithm}[ht!]
\caption{Multiplicative updates for polynomial root finding} \label{muroot}
\begin{algorithmic}[1]
\REQUIRE The polynomial $f(x) = p(x)-q(x)$ where $p$ and $q$ have nonnegative coefficients, 
an initial point $x_0 > 0$ with $x_0 \in [r_k, r_{k+1}]$ 
where $r_k$ is the $k$th nonnegative real root of $f$ (where $r_0 = 0$ and $r_{m+1} = +\infty$). 

\ENSURE 
If $f$ satisfies Assumption~\ref{as1}, the sequence $x_t$ (resp.\@ $x_{-t}$) 
converges to $r_{k+1}$ (resp.\@ $r_k$) if $p(x_0) > q(x_0)$, to $r_{k}$ (resp.\@ $r_{k+1}$) otherwise.   \medskip

\FOR{$t = 0, 1, 2, \dots$}

\STATE $x_{t+1} = x_{t} \frac{ p(x_t) }{ q(x_t) }$.

\STATE $x_{-(t+1)} = x_{-t} \frac{ q(x_{-t}) }{ p(x_{-t}) }$. 

\ENDFOR

\end{algorithmic}
\end{algorithm}

\begin{theorem} \label{convr}
Under the same assumptions as in Theorem \ref{MainTh}, and assuming without loss of generality that $p(x_0) > q(x_0)$ 
for $r_k < x_0 < r_{k+1}$, 
the sequences $\{x_t\}_{t\geq 1}$ and $\{x_t\}_{t\leq -1}$ generated by Algorithm~\ref{muroot} 
converge to $r_{k+1}$ and $r_{k}$, respectively.  
\end{theorem}
\begin{proof} Let us focus on the sequence $\{x_t\}_{t\geq 1}$; the same proof holds for $\{x_t\}_{t\leq -1}$. 
We have $p(x_0) > q(x_0)$ since $x_0$ is not a root of $f$ by assumption. 
By Theorem~\ref{MainTh},    
\[ 
x_0 < x_1 < x_2 < \dots \leq r_{k+1}. 
\]  
Therefore, $\{x_t\}_{t\geq 1}$ must converge to a limit point $s$ (possibly $+\infty$ if $k=m$). 
Suppose $s < r_{k+1}$. For any $x \in  [x_0,s]$, we have 
\[
 x \frac{p(x)}{q(x)} \geq x L, 
\; \textrm{ with } \;   L = \min_{x \in [x_0,s]} \frac{p(x)}{q(x)} > 1, 
\]
since $p(x) > q(x)$ for all $x \in [x_0, s] \subset ]r_k,r_{k+1}[$. 
By construction, we therefore have $x_0 L^t \leq x_t \leq s < +\infty$ for all $t \geq 1$ which is a contradiction since $L > 1$.  
\end{proof}


\begin{remark}
For Theorems~\ref{MainTh} and \ref{convr} to hold, 
the decomposition $f(x) = p(x)-q(x)$ can be chosen differently as in~\eqref{pxnx} as long as $p$ and $q$ have nonnegative coefficients. In fact, for any polynomial $d(x)$ with nonnegative coefficients, we can use the decomposition $f(x) = \big(p(x)+d(x)\big) - \big(q(x)+d(x)\big)$ which will simply make Algorithm converge slower since $\frac{p(x)+d(x)}{q(x)+d(x)}$ will be closer to 1 than $\frac{p(x)}{q(x)}$. 
\end{remark}

The simplest case for which Theorems~\ref{MainTh} and \ref{convr} apply is when $f(x) = x - b$ for $b > 0$. For $x_0 < b$ (resp.\@ $x_0 > b$), 
the updates are given by 
\[
x_1 = x_0 \frac{x_0}{b} 
\quad \text{ and } \quad 
x_{-1} = x_0 \frac{b}{x_0} = b, 
\]
so that $x_{-1}$ converges in one step to the root of $f$ while $\{x_t\}_{t \geq 0} = x_0 \left( \frac{x_0}{b} \right)^{2^t - 1}$ 
converges to zero (resp.\@ infinity) quadratically.  

For higher degree polynomials, the convergence is linear, as shown in the theorem below.

\begin{theorem} \label{thconv}
If $f$ satisfies Assumption~\ref{as1}, Algorithm~\ref{muroot} asymptotically converges linearly to simple roots of $f$. 
\end{theorem} 
\begin{proof}
Let us focus on the one-point iteration $F(x) = x \frac{p(x)}{q(x)}$ where 
\begin{itemize}
\item the initial point $x_0$ is smaller but sufficiently close to the simple root $\alpha$, that is, $\alpha - \delta < x_0 < \alpha$ for some $\delta > 0$,  
\item $p(x_0) > q(x_0) \iff f(x_0) > 0$. 
\end{itemize} 
The other cases can be treated in a similar way. 

Since $\alpha$ is a simple root of $f$, 
we have $f'(\alpha) \neq 0$ hence $p'(\alpha) < q'(\alpha)$ for $\delta$ sufficiently small, 
since $p(x_0) > q(x_0)$ and $p(\alpha) = q(\alpha)$.  
By Lagrange mean value theorem, we have 
\begin{equation} \label{xtp1}
x_{t+1} = F(x_t) = \alpha + (x_t-\alpha) F'(\zeta) \quad \text{ for some } \zeta \in [x_t,\alpha]. 
\end{equation}
By Theorem~\ref{MainTh}, $x_0 < x_1 < \dots \leq \alpha$, 
so that the error $e_t$ at the $t$th step of Algorithm~\ref{muroot} satisfies $e_{t} = \alpha - x_t \geq 0$. 
Injecting $e_t$ in~\eqref{xtp1}, we obtain $e_{t+1} = F'(\zeta) e_t$. 
If we show that $0 \leq \ell = \min_{\alpha - \delta < \zeta < \alpha} F'(\zeta) < 1$ for $\delta$ sufficiently small, the proof is complete since this implies a linear convergence rate of ratio $\ell < 1$.   
First, $e_t \geq 0$ for all $t$ implies that $F'(\zeta) \geq 0$. 
Second, recall that since $q$ is a polynomial with nonnegative coefficients and at least one positive coefficient (by Assumption~\ref{as1}),  
$q(x) > 0$ for all $x > 0$ hence $F(x)$ is differentiable for all $x > 0$. 
We compute 
\begin{align}
F'(\alpha) 
& = \frac{p(\alpha)}{q(\alpha)} 
+ \alpha 
\frac{p'(\alpha)q(\alpha)-q'(\alpha)p(\alpha)}{q^2(\alpha)} =  1 - \alpha  \frac{q'(\alpha)-p'(\alpha)}{q(\alpha)},  \label{rateconv}
\end{align}
since $p(\alpha) = q(\alpha)$. 
We have $F'(\alpha) < 1$ since $\alpha > 0$, $q'(\alpha) > p'(\alpha)$ and $q(\alpha) > 0$. 
\end{proof}

Note that the convergence cannot be in general faster than linear since $F(x)$ has order one, where the order $p$ of a one-point iteration $F$ is defined as~\cite[p.344, Theorem~8.1]{Ralts} 
\[
F(\alpha) = \alpha; \quad  F^{(j)}(\alpha) = 0, 0 \leq j < p; \quad F^{(p)}(\alpha) \neq 0, 
\] 
with $F^{(j)}$ the $j$th derivative of $F$. The multiplicative updates have order one at simple roots of $f$.

\section{Numerical Example}  \label{examp}

Let us consider 
\[ 
f(x)  = (x-1)(x-2)(x-3)(x-1+i)(x-1-i) = x^5-8x^4+25x^3-40x^2+34x-12, 
\] 
with $p(x)  =  x^5+25x^3+34x$ 
and $q(x) = 8x^4+40x^2+12$, 
for which $r_0 = 0$, $r_1 = 1$, $r_2 = 2$, $r_3 = 3$, $r_4 = +\infty$, $r_5 = 1+i$, $r_6 = 1-i$. 
Figure~\ref{functex} displays the polynomial (on the left) along with the evolution of the iterates generated by Algorithm~\ref{muroot} using $x_0 = 2.5$ (on the right). 
\begin{figure}[h]
\begin{center}
\includegraphics[width=0.8\textwidth]{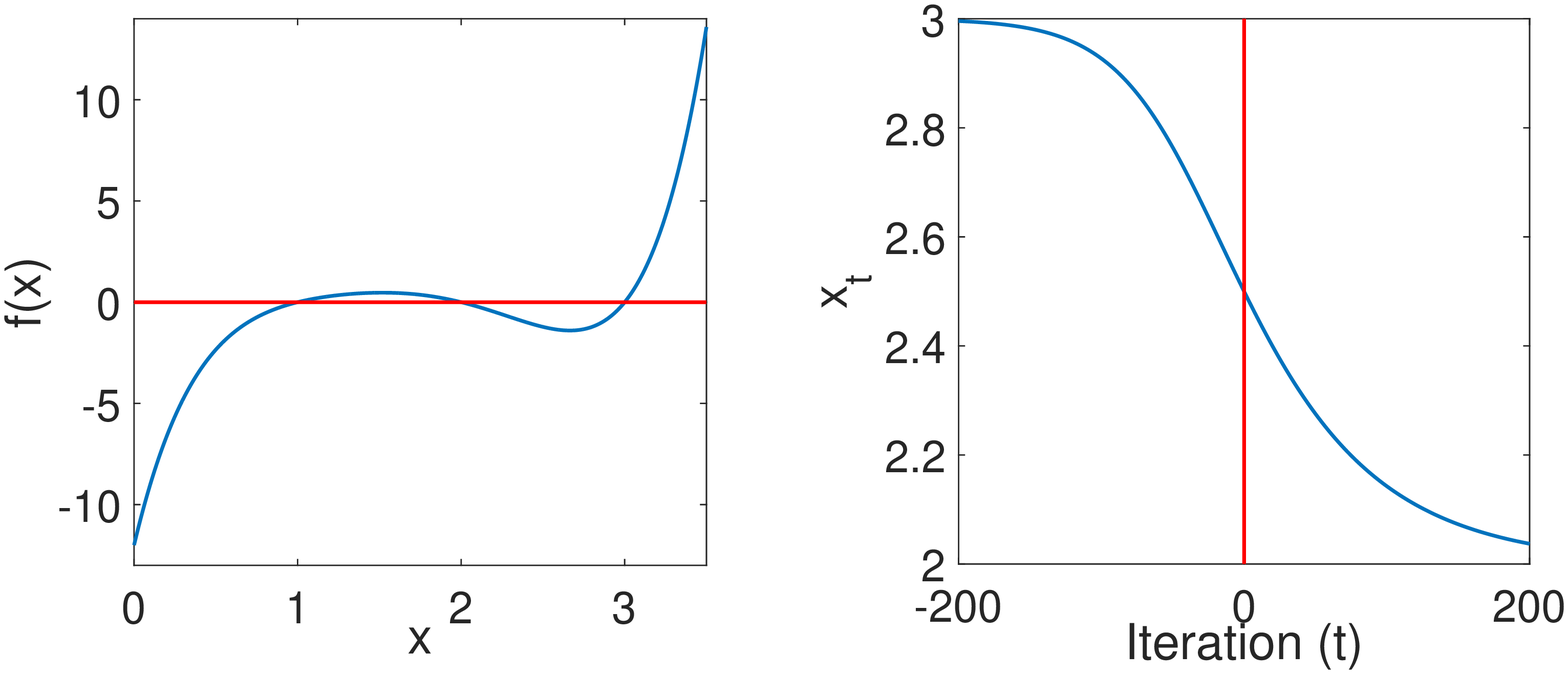}
\caption{(Left) Polynomial $f(x)$. (Right) Evolution of the iterates under the multiplicative updates.} 
\label{functex}
\end{center}
\end{figure}
As proved in Theorems~\ref{MainTh} and~\ref{convr}, the iterates remain in the interval $[2,3]$ and converge to the bounds of this interval. 
In this example, $p(x_0) < q(x_0)$ hence the sequence $\{x_t\}_{t \geq 0}$ generated by Algorithm~\ref{muroot} converges to 2 and $\{x_t\}_{t \leq 0}$ to 3. 

Figure~\ref{linconv} illustrates the linear convergence of the updates as proved in Theorem~\ref{thconv}. 
\begin{figure}[h]
\begin{center}
\includegraphics[width=0.4\textwidth]{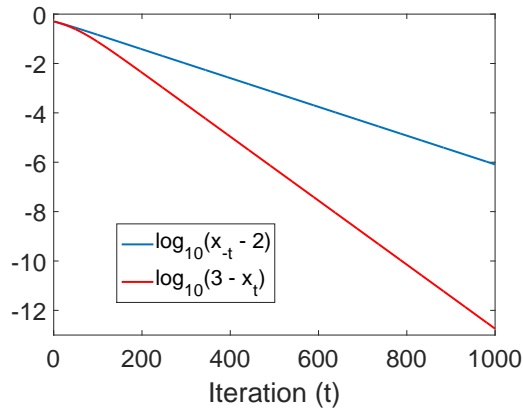}
\caption{Evolution of the logarithm of the error of the iterates under the multiplicative updates.} 
\label{linconv}
\end{center}
\end{figure}  
For the root $r_3=3$, the asymptotic rate of convergence from~\eqref{rateconv} is given by 
\[ 
F'(3) = 1 - \alpha  \frac{q'(\alpha)-p'(\alpha)}{q(\alpha)} 
= 
1 - 3  \frac{q'(3)-p'(3)}{q(3)} 
= 0.9706, 
\]
so that Algorithm~\ref{muroot} (asymptotically) requires about 77 iterations ($F'(3)^{77} \approx 0.1$) to gain one digit of accuracy. 
For the root $r_2 = 2$, we obtain 
\[
1 - \alpha  \frac{p'(\alpha)-q'(\alpha)}{q(\alpha)} 
= 
1 - 2  \frac{p'(2)-q'(2)}{q(2)} 
= 0.9867, 
\]
so that about 170 iterations are necessary to gain one digit of accuracy. 


\section{Discussion} 

In this paper, we analyzed simple multiplicative updates to find the nonnegative real roots of a polynomial. 
We proved a rather surprising fact that under the assumption that the roots of $f$ have nonnegative real parts (Assumption~\ref{as1}), 
the updates always remain in the same interval between two real roots and monotonically converge to these roots. 
These updates converge relatively slowly (linearly for simple roots). 
The main motivation to study these updates came from a vast body of literature using such updates for matrix factorization problems with nonnegativity constraints. However, it is unlikely for these updates to be competitive for polynomial root finding as it is a highly studied problem for which there exist more general and much more efficient methods. 
However, it would be an interesting direction for further research to analyze acceleration schemes, and use these schemes in practical applications from the nonnegative matrix factorization literature~\cite{cichocki2009nonnegative}.  
For example, we observed that shifting the polynomial $f$ can accelerate convergence significantly (as the ratio between $p$ and $q$ goes away from one).


\section*{Acknowledgements}  

We are grateful to the reviewer that helped us improve the paper significantly. 

The author acknowledges the support of the ERC (starting grant n$^\text{o}$ 679515) and of the F.R.S.-FNRS (incentive grant for scientific research n$^\text{o}$ F.4501.16).

\bibliographystyle{spmpsci}  

\bibliography{MUroot} 

\begin{thebibliography}{1}
\providecommand{\url}[1]{{#1}}
\providecommand{\urlprefix}{URL }
\expandafter\ifx\csname urlstyle\endcsname\relax
  \providecommand{\doi}[1]{DOI~\discretionary{}{}{}#1}\else
  \providecommand{\doi}{DOI~\discretionary{}{}{}\begingroup
  \urlstyle{rm}\Url}\fi

\bibitem{cichocki2009nonnegative}
Cichocki, A., Zdunek, R., Phan, A., Amari, S.i.: Nonnegative matrix and tensor
  factorizations: applications to exploratory multi-way data analysis and blind
  source separation.
\newblock John Wiley \& Sons (2009)

\bibitem{G14a}
Gillis, N.: {The Why and How of Nonnegative Matrix Factorization}.
\newblock In: J.~Suykens, M.~Signoretto, A.~Argyriou (eds.) Regularization,
  Optimization, Kernels, and Support Vector Machines, pp. 257--291. Chapman \&
  Hall/CRC, Machine Learning and Pattern Recognition Series (2014)

\bibitem{lee1999learning}
Lee, D., Seung, H.: Learning the parts of objects by non-negative matrix
  factorization.
\newblock Nature \textbf{401}(6755), 788--791 (1999)

\bibitem{lee2001algorithms}
Lee, D., Seung, H.: Algorithms for non-negative matrix factorization.
\newblock In: Advances in neural information processing systems, pp. 556--562
  (2001)

\bibitem{Ralts}
Raltson, A., Rabinowitz, P.: A first Course in Numerical Analysis.
\newblock McGraw-Hill, Inc., Singapore (1978)

\bibitem{sha2007multiplicative}
Sha, F., Lin, Y., Saul, L., Lee, D.: Multiplicative updates for nonnegative
  quadratic programming.
\newblock Neural computation \textbf{19}(8), 2004--2031 (2007)

\end{thebibliography}

\end{document}